\documentclass{amsart}

\usepackage{amsmath}
\usepackage{amsfonts}
\usepackage{amssymb}
\usepackage{color}
\usepackage{euscript}
\usepackage{mathrsfs}

\usepackage{hyperref}

\newtheorem{theorem}{Theorem}
\newtheorem{remark}[theorem]{Remark}
\newtheorem{lemma}[theorem]{Lemma}

\newtheorem{corollary}[theorem]{Corollary}
\newtheorem{definition}[theorem]{Definition}

\newtheorem{prop}[theorem]{Proposition}

\newtheorem{notation}[theorem]{Notation}
\numberwithin{equation}{section}
\numberwithin{theorem}{section}

\newcommand{\norm}[1]{\| #1 \|}

\newcommand{\inprod}[1]{\ensuremath{\langle #1\rangle}}

\title[Smoothness of heat kernel measures on infinite Heisenberg groups]
{Smoothness of heat kernel measures on infinite-dimensional
Heisenberg-like groups}
\author[Dobbs]{Daniel Dobbs}
\address{Department of Mathematics\\
University of Virginia \\
Charlottesville, VA 22903 USA} \email{dwd2r@virginia.edu}
\author[Melcher]{Tai Melcher{$^*$}}
\thanks{\footnotemark {$^*$} This research was supported in part by NSF
Grant DMS-0907293.}
\address{Department of Mathematics\\
University of Virginia\\ Charlottesville, VA 22903 USA}
\email{melcher@virginia.edu}

\keywords{Heisenberg group, heat kernel measures, smooth measures, integration
by parts}
\subjclass[2010]{Primary  58J65; 
Secondary 35B65, 
35R15} 

\begin{document}

\begin{abstract}
We study measures associated to Brownian motions on infinite-dimensional
Heisenberg-like groups. In particular, we prove that the associated
path space measure and heat kernel measure satisfy a strong definition
of smoothness.  
\end{abstract}

\maketitle
\tableofcontents

\section{Introduction}

Recall that a measure $\mu$ on $\mathbb{R}^n$ is {\it smooth} if
$\mu$ is absolutely continuous with respect to Lebesgue measure and
the associated density is a smooth function on $\mathbb{R}^n$.  
If one wishes
to generalize this notion of smoothness of measure to an infinite-dimensional
space, one immediately encounters complications due to the lack of an infinite-dimensional Lebesgue
measure.  Thus, we consider the following more intrinsic
definition of smoothness for a measure on $\mathbb{R}^n$:  
for any multi-index
$\alpha=(\alpha_1,\ldots,\alpha_n)\in\{0,1,2,\ldots\}^n$,
there exists a function $z_\alpha\in C^\infty(\mathbb{R}^n)\cap
L^{\infty-}(\mu)$ such that
\[ \int_{\mathbb{R}^n} \partial^\alpha f\,d\mu 
	=\int_{\mathbb{R}^n} fz_\alpha\,d\mu, \quad \text{ for all } f\in
		C_c^\infty(\mathbb{R}^n), \]
where $L^{\infty-}:=\cap_{p\ge1} L^p$ and $\partial^\alpha=\prod_{i=1}^n \partial_i^{\alpha_i}$.
This definition of smoothness is in fact equivalent to our
first understanding (see for example \cite{Driver2003}), 
and it is obviously better suited to adapt to
infinite dimensions and the absence of a canonical reference measure.

In the present paper we adapt the above definition to give a direct proof of the
smoothness of elliptic heat kernel measures on infinite-dimensional
Heisenberg-like groups.  Typically, it is not possible to verify that a measure
on an infinite-dimensional space is smooth in this way and much
weaker interpretations must be made; see for example
\cite{BaudoinTeichmann2005,Malliavin1990a,MattinglyPardoux2006}.

Let $G$ be an infinite-dimensional
Heisenberg-like group, $\mathfrak{g}_{CM}$ be its Cameron-Martin Lie
subalgebra, and $\{\xi_t\}_{t\ge0}$ be a Brownian motion
on $G$ (see Section \ref{s.setup} for definitions). Then we have the following theorem.

\begin{theorem} 
\label{t.thm2}
Fix $T>0$, and let $m\in\mathbb{N}$ and $h_1,\ldots,h_m\in
\mathfrak{g}_{CM}$.  Then there exist $\tilde{z},\hat{z}\in L^{\infty-}$
depending on $h_1,\ldots,h_m$ such that, for any suitably nice function $f$ on $G$,
\[ \mathbb{E}\left[(\tilde{h}_1\cdots\tilde{h}_m f)(\xi_T)\right]
	= \mathbb{E}[f(\xi_T)\tilde{z}]
\quad \text{ and } \quad
\mathbb{E}\left[(\hat{h}_1\cdots\hat{h}_m f)(\xi_T)\right]
	= \mathbb{E}[f(\xi_T)\hat{z}], \]
where $\tilde{h}$ and $\hat{h}$ are the left and right invariant vector
fields, respectively, associated to $h\in\mathfrak{g}_{CM}$.
\end{theorem}

This result is proved by first establishing smoothness results for the induced measure
on the associated path space.  In particular, let $\mathcal{W}_T(G)$ denote
continuous path space on $G$ and $\mathcal{H}_T(\mathfrak{g}_{CM})$
denote the space of absolutely continuous paths on $\mathfrak{g}_{CM}$
with finite energy (see Notation \ref{n.path}).  Then we prove the following theorem.

\begin{theorem} 
\label{t.thm1}
Let $m\in\mathbb{N}$ and $\mathbf{h}_1,\ldots,\mathbf{h}_m\in
\mathcal{H}_T(\mathfrak{g}_{CM})$.  Then there exists $\hat{Z}\in L^{\infty-}$ depending on
$\mathbf{h}_1,\ldots,\mathbf{h}_m$ 
such that, for any suitably nice function $F$ on $\mathcal{W}_T(G)$, 
\[ \mathbb{E}\left[(\hat{\mathbf{h}}_1\cdots\hat{\mathbf{h}}_m F)(\xi)\right]
	= \mathbb{E}[F(\xi)\hat{Z}], \]
where $\hat{\mathbf{h}}$ is the right invariant vector field
associated to $\mathbf{h}\in\mathcal{H}_T(\mathfrak{g}_{CM})$.
\end{theorem}

Theorem \ref{t.thm1} is stated more precisely and proved in Theorem \ref{t.pathIBP}; 
Theorem \ref{t.thm2} is the content of Theorem
\ref{t.rHeatKernelIBP} and Corollary \ref{c.lHeatKernelIBP}.
Note that these theorems give a strong satisfaction of smoothness for measures in
infinite dimensions.  

The organization of the paper is as follows. Section \ref{s.setup} recalls
the definitions of infinite-dimensional Heisenberg-like groups and Brownian
motions on these groups, first studied in \cite{DriverGordina2008}.  In
Section \ref{s.ps}, we recall the quasi-invariance and first-order integration
by parts results proved in \cite{DriverGordina2008} for the
path space measure, and, building on these results, give the
integration by parts formulae that prove Theorem \ref{t.thm1}.  In Section
\ref{s.gp}, we show how these path space results immediately give integration by
parts formulae for heat kernel measures on the group.

Finally, let us here mention some references to other quasi-invariance and
integration by parts results for measures in infinite-dimensional
curved settings; see
\cite{Airault2006,Albeverio1997,Bell2006,Driver1997,Fang1999,Hsu2009} and
their references.

{\it Acknowledgement.} The authors would like to thank Bruce Driver for
suggesting this problem.  We would also like to thank the anonymous referee
for thoughtful recommendations to improve the readability of this
paper.

\section{Brownian motion on infinite-dimensional Heisenberg-like groups}
\label{s.setup}

In this section, we recall the definitions of infinite-dimensional
Heisenberg-like groups and Brownian motion on these spaces.  For
more details on this construction, see Sections 2 and 4 of
\cite{DriverGordina2008}.  One may also consult this reference for
motivating examples, including the
finite-dimensional Heisenberg groups as well as the Heisenberg group
of a symplectic vector space.

Let $(W,H,\mu)$ denote an abstract Wiener space; that is, $W$ is a real
separable Banach space equipped with Gaussian measure $\mu$ and $H$ is the
associated Cameron-Martin subspace. Let $\mathbf{C}$
be a real vector space with inner product
$\langle\cdot,\cdot\rangle_\mathbf{C}$ and
$\mathrm{dim}(\mathbf{C})=:N<\infty$.  Let $\omega:W\times
W\rightarrow\mathbf{C}$  be a continuous skew-symmetric bilinear
form on $W$. 

\begin{definition}
Let $\mathfrak{g}$ denote $W\times\mathbf{C}$ when thought of as a Lie algebra
with the Lie bracket given by
\begin{equation}
\label{e.3.5}
[(X_1,V_1), (X_2,V_2)] := (0, \omega(X_1,X_2)).
\end{equation}
We may also equip $W\times\mathbf{C}$ with the group multiplication
given by
\begin{equation}
\label{e.3.2}
(w_1,c_1)\cdot(w_2,c_2) = \left( w_1 + w_2, c_1 + c_2 +
    \frac{1}{2}\omega(w_1,w_2)\right).
\end{equation}
We will denote $W\times\mathbf{C}$ by $G$ when thought of as a group, and 
we will call $G$ constructed in this way a {\em Heisenberg-like group}.
\end{definition}
It is easy to verify that, given this bracket and multiplication,
$\mathfrak{g}$ is indeed a Lie algebra and $G$ is a group
with $g^{-1}=-g$ and identity $e=(0,0)$.

\begin{notation}
Let $\mathfrak{g}_{CM}$ denote $H\times\mathbf{C}$ when thought of as a Lie
subalgebra of $\mathfrak{g}$, and we will refer to $\mathfrak{g}_{CM}$ as the
{\em Cameron-Martin subalgebra} of $\mathfrak{g}$.
\end{notation}

The space $\mathfrak{g}=G=W\times\mathbf{C}$ is a Banach space with the norm
\[ \|(w,c)\|_{\mathfrak{g}} := \|w\|_W + \|c\|_\mathbf{C}, \]
and $\mathfrak{g}_{CM}=H\times\mathbf{C}$ is a Hilbert space with respect to the inner
product
\[ \langle (A,a),(B,b)\rangle_{\mathfrak{g}_{CM}}
	:= \langle A,B\rangle_H + \langle a,b\rangle_\mathbf{C}. \]
The associated Hilbertian norm on $\mathfrak{g}_{CM}$ is given by
\[ \|(A,a)\|_{\mathfrak{g}_{CM}} := \sqrt{\|A\|_H^2 + \|a\|_\mathbf{C}^2}. \]

Let $i:H\rightarrow W$ denote the inclusion map, $i^*:W^*\rightarrow
H^*$ denote its transpose, and $H_*:=\{h\in H:
\langle\cdot,h\rangle_H\in\mathrm{Range}(i^*)\}$.  Let
$\{B_t,B^0_t\}_{t\ge0}$ be a Brownian motion on $\mathfrak{g}$ with
variance determined by
\[
\mathbb{E}\left[\langle (B_s,B^0_s),(A,a)\rangle_{\mathfrak{g}_{CM}} \langle
	(B_t,B^0_t),(C,c)\rangle_{\mathfrak{g}_{CM}}\right]
    = \langle (A,a),(C,c) \rangle_{\mathfrak{g}_{CM}} \min(s,t),
\]
for all $s,t\ge0$, $A,C\in H_*$, and $a,c\in\mathbf{C}$. 

\begin{definition}
\label{d.bm}
The continuous $G$-valued process given by
\begin{equation}
\label{e.bm} 
\xi_t 
	= \left( B_t, B^0_t + \frac{1}{2}\int_0^t \omega(B_s,dB_s)\right)
\end{equation}
is a {\em Brownian motion} on $G$.
For $T>0$, let $\nu_T=\mathrm{Law}(\xi_{T})$ denote the {\em heat kernel measure
at time $T$} on $G$.
\end{definition}

Proposition 4.1 of \cite{DriverGordina2008} gives
details on how the above stochastic integral is defined, and 
more generally that reference proves many properties of the process $\xi_t$ and its
distribution.  In particular, in Corollary 4.9 of that reference it is proved that $\nu_T$ is
invariant under the inversion map $g\mapsto g^{-1}$; that is,
for any $T>0$, 
\begin{equation} 
\label{e.inv}
\mathbb{E}[f(\xi_T)] 
	= \int_G f(g)\,d\nu_T(g) = \int_G f(g^{-1})\,d\nu_T(g)
	= \mathbb{E}[f(\xi_T^{-1})].
\end{equation}

\section{The path space measure}
\label{s.ps}

In this section, we prove that 
$\nu=\mathrm{Law}(\xi)$ satisfies its own strong smoothness properties.

\begin{notation}
\label{n.path}
Fix $T>0$. For a Banach space $X$, let 
\[ \mathcal{W}_T(X) := \{x:[0,T]\rightarrow X: x \text{ continuous and } x(0)=0 \}\]
equipped with the sup norm
topology, and, for a Hilbert space $K$, let $\mathcal{H}_T(K)$ denote the
absolutely continuous paths in $\mathcal{W}_T(K)$ with finite energy.
In particular, for $X=G$ 
\[ \|\mathbf{g}\|_{\mathcal{W}_T(G)} 
	:= \sup_{0\le t\le T} \|\mathbf{g}(t)\|_\mathfrak{g}
	= \sup_{0\le t\le T} \left(\|\mathbf{w}(t)\|_W +
		\|\mathbf{c}(t)\|_\mathbf{C}\right) \]
for all $\mathbf{g}=(\mathbf{w},\mathbf{c})\in \mathcal{W}_T(G)$,
and for $K=\mathfrak{g}_{CM}$ 
\[ \|\mathbf{h}\|_{\mathcal{H}_T(\mathfrak{g}_{CM})}^2 
	:= \int_0^T \|\dot{\mathbf{h}}(t)\|_{\mathfrak{g}_{CM}}^2\,dt
	= \int_0^T \left(\|\dot{\mathbf{A}}(t)\|_{H}^2 
		+ \|\dot{\mathbf{a}}(t)\|_{\mathbf{C}}^2\right)\,dt \] 
for all $\mathbf{h}=(\mathbf{A},\mathbf{a})\in
\mathcal{H}_T(\mathfrak{g}_{CM})$.
\end{notation}

\begin{remark}
\label{r.fernique}
Recall that, for $\{B_t\}_{t\ge0}$ Brownian motion on $W$, $\mathrm{Law}(B)$ is a
Gaussian measure on the separable Banach space $\mathcal{W}_T(W)$.  
Thus, by
Fernique's theorem (see for example Theorem 3.1 of
\cite{Kuo1975}), there exists $\delta_0>0$ such that for all
$\delta<\delta_0$
\[ \mathbb{E}\left[\exp(\delta\|B\|_{\mathcal{W}_T(W)}^2)\right]<\infty. \]
Additionally, in Proposition 4.1 of \cite{DriverGordina2008}, it is proved that for any
$p\in[1,\infty)$
\[ \mathbb{E}\left\|\int_0^\cdot
	\omega(B_s,dB_s) \right\|_{\mathcal{W}_T(\mathbf{C})}^p<\infty. 
\]
\end{remark}

The following theorem is a slight generalization of Theorem 5.2 in
\cite{DriverGordina2008}, and the proof is analogous.
\begin{theorem}
\label{t.qi}
Let $\mathbf{h}=(\mathbf{A},\mathbf{a})\in\mathcal{H}_T(\mathfrak{g}_{CM})$.  If
$F,Z:\mathcal{W}_T(G)\rightarrow[0,\infty]$ are measurable functions, then
\begin{equation}
\label{e.qi} 
\mathbb{E}[F(\mathbf{h}\cdot \xi)Z(B,B^0)]
    = \mathbb{E}[F(\xi)Z(B-\mathbf{A},B^0-\mathbf{a}-u_\mathbf{A})J_\mathbf{h}], 
\end{equation}
where 
\begin{equation}
\label{e.uA} 
u_\mathbf{A}(t) 
	:= \frac{1}{2}\int_0^t\omega(\mathbf{A}(s)-2B_s,\dot{\mathbf{A}}(s))\,ds
	\in\mathcal{H}_T(\mathbf{C}) 
\end{equation}
and $J_\mathbf{h} = J_\mathbf{h}(B,B^0)$ is given by
\begin{multline}
\label{e.Jh}
J_\mathbf{h}
    := \exp\bigg\{ \int_0^T \langle \dot{\mathbf{A}}(t),dB_t\rangle_H
        + \left\langle \dot{\mathbf{a}}(t) +
            \frac{1}{2}\omega(\mathbf{A}(t)-2B_t,
		\dot{\mathbf{A}}(t)),dB_t^0\right\rangle_\mathbf{C} \\
    - \frac{1}{2}\int_0^T \left(\|\dot{\mathbf{A}}(t)\|_H^2
        + \left\|\dot{\mathbf{a}}(t) 
	+ \frac{1}{2}\omega(\mathbf{A}(t)-2B_t,\dot{\mathbf{A}}(t))
            \right\|_\mathbf{C}^2 \right)\,dt\bigg\}.
\end{multline}
Moreover, equation (\ref{e.qi}) holds for all measurable
$F,Z:\mathcal{W}_T(G)\rightarrow\mathbb{R}$ such that
\[ \mathbb{E}|F(\mathbf{h}\cdot \xi)Z(B,B^0)|
	= \mathbb{E}|F(\xi)Z(B-\mathbf{A},B^0-\mathbf{a}-u_\mathbf{A})J_\mathbf{h}|
	<\infty. \]
\end{theorem}

\begin{proof}
First combining (\ref{e.3.2}) and (\ref{e.bm}) gives
\begin{multline*}
\mathbb{E}[ F(\mathbf{h}\cdot \xi)Z(B,B^0) ] \\
	=  \mathbb{E}\left[ F\left(B+\mathbf{A},B^0+\mathbf{a}+\frac{1}{2}\int_0^\cdot
		\omega(B_s,dB_s)+\frac{1}{2}\omega(\mathbf{A},B)\right)Z(B,B^0)
\right].
\end{multline*} 
Now translating $(B,B^0)\mapsto (B-\mathbf{A},B^0-\mathbf{a})$ and applying
the standard Cameron-Martin theorem (see for example Theorem 1.2 of Chapter II
of \cite{Kuo1975}) implies that
\begin{align*}
\mathbb{E}&[ F(\mathbf{h}\cdot \xi)Z(B,B^0) ] \\
	&=  \mathbb{E}\bigg[ F\left(B,B^0 + \frac{1}{2}\int_0^\cdot
		\omega(B_s-\mathbf{A}(s),d(B_s-\mathbf{A}(s)))
		+\frac{1}{2}\omega(\mathbf{A},B-\mathbf{A})\right) \\
	&\qquad \times Z(B-\mathbf{A},B^0-\mathbf{a})\bar{J}_\mathbf{h}(B,B_0)\bigg] 
\end{align*} 
where $\bar{J}_\mathbf{h} = \bar{J}_\mathbf{h} (B,B^0)$ is given by
\begin{multline*}
\bar{J}_\mathbf{h} 
	:= \exp\left(\int_0^T \langle\dot{\mathbf{A}}(t),dB_t\rangle_H 
		- \frac{1}{2}\int_0^T \|\dot{\mathbf{A}}(t)\|_H^2\,dt\right) \\
	\times\exp\left(\int_0^T \langle\dot{\mathbf{a}}(t),dB_t^0\rangle_\mathbf{C} 
		- \frac{1}{2}\int_0^T
\|\dot{\mathbf{a}}(t)\|_\mathbf{C}^2\,dt\right).
\end{multline*}
This may be rewritten as
\begin{align*}
\mathbb{E}[ F(\mathbf{h}&\cdot \xi)Z(B,B^0) ] \\
	&= \mathbb{E}\bigg[ F\left(B,B^0 + \frac{1}{2}\int_0^\cdot
		\omega(B_s,dB_s) 
		+ \frac{1}{2}\int_0^\cdot \omega(\mathbf{A}(s)-2B_s,\dot{\mathbf{A}}(s))\ ds\right)  \\
	&\qquad \times
		Z(B-\mathbf{A},B^0-\mathbf{a})\bar{J}_\mathbf{h}(B,B_0)\bigg].
\end{align*} 
Freezing integration over $B$ (that is, using Fubini) and translating again, 
this time $B_0\mapsto B_0-u_\mathbf{A}$ with $u_A$ as defined in (\ref{e.uA}),
we may again apply the Cameron-Martin theorem to get that
\[
\mathbb{E}[F(\mathbf{h}\cdot \xi)Z(B,B^0)]
	=\mathbb{E}\left[ F(\xi)Z(B-\mathbf{A},B^0-\mathbf{a}-u_\mathbf{A}) 
		\bar{J}_\mathbf{h}(B,B^0-u_\mathbf{A})
		\bar{J}_{(0,u_\mathbf{A})}\right].
\]
Now one may simplify to show that
\[  \bar{J}_\mathbf{h}(B,B^0-u_\mathbf{A})\bar{J}_{(0,u_\mathbf{A})} = J_\mathbf{h}, \]
where $J_\mathbf{h}$ is as defined in \eqref{e.Jh}.
\end{proof}

\begin{remark}
If we take $Z\equiv 1$ in the previous theorem, this is the statement
that $\nu=\mathrm{Law}(\xi)$ is quasi-invariant under left translation
by elements of $\mathcal{H}_T(\mathfrak{g}_{CM})$. It is worth
recalling that the above proof fails for right translation, as the
requisite translating element in that case is not absolutely
continuous and thus the Cameron-Martin theorem is no longer available;
see Remark 5.3 of \cite{DriverGordina2008} for details.  
\end{remark}

We now have a few technical estimates and notations that will allow us to
prove the desired integration by parts formulae in Theorem \ref{t.pathIBP}.  
The following result is a restatement of Proposition 5.4 of
\cite{DriverGordina2008}.  We include the proof here for completeness.
\begin{prop} 
\label{p.Jhp}
Let $p\in[1,\infty)$.  Then there exists $\kappa=\kappa(p)>0$ such that, for
all $\mathbf{h}\in \mathcal{H}_T(\mathfrak{g}_{CM})$ such that
$\|\mathbf{h}\|_{\mathcal{H}_T(\mathfrak{g}_{CM})}<\kappa$,
\[ \mathbb{E}[J_\mathbf{h}(B,B^0)^p] <\infty. \]
\end{prop}

\begin{proof}
For the purpose of this proof, let $\mathbb{E}_{B^0}$ and $\mathbb{E}_B$
denote expectation relative to $B^0$ and $B$, respectively.  We may write
\begin{align*} 
J_\mathbf{h}(B,B^0)^p
	&= \exp\left\{  p\int_0^T 
        \left\langle \dot{\mathbf{a}}(t) +
		\frac{1}{2}\omega(\mathbf{A}(t)-2B_t,\dot{\mathbf{A}}(t)),dB_t^0\right\rangle_\mathbf{C}
		\right\} \\
	&\quad\times \exp\left\{ p\int_0^T \langle \dot{\mathbf{A}}(t),dB_t\rangle_H
		- \frac{1}{2}p\int_0^T \|\dot{\mathbf{A}}(t)\|_H^2 \,dt \right\} \\
	&\quad \times \exp\left\{ 
		- \frac{1}{2}p\int_0^T
        \left\|\dot{\mathbf{a}}(t) +
		\frac{1}{2}\omega(\mathbf{A}(t)-2B_t,\dot{\mathbf{A}}(t))
            \right\|_\mathbf{C}^2 \,dt \right\}.
\end{align*}
Since
\begin{multline*}
\mathbb{E}_{B^0}\left[ \exp\left\{  p\int_0^T 
        \left\langle \dot{\mathbf{a}}(t) +
		\frac{1}{2}\omega(\mathbf{A}(t)-2B_t,\dot{\mathbf{A}}(t)),
		dB_t^0\right\rangle_\mathbf{C} \right\}\right] \\
	= \exp\left\{ 
		\frac{1}{2}p^2\int_0^T
        \left\|\dot{\mathbf{a}}(t) +
		\frac{1}{2}\omega(\mathbf{A}(t)-2B_t,\dot{\mathbf{A}}(t))
            \right\|_\mathbf{C}^2 \,dt \right\},
\end{multline*}
we may write $\mathbb{E}_{B^0}[J_\mathbf{h}(B,B^0)^p]= UV$, where
\[ U := \exp\left\{ p\int_0^T \langle \dot{\mathbf{A}}(t),dB_t\rangle_H
		- \frac{1}{2}p\int_0^T \|\dot{\mathbf{A}}(t)\|_H^2\, dt\right\} \]
and
\[ V := \exp\left\{\frac{1}{2}(p^2-p)\int_0^T \left\|\dot{\mathbf{a}}(t) +
		\frac{1}{2}\omega(\mathbf{A}(t)-2B_t,\dot{\mathbf{A}}(t))
		\right\|_\mathbf{C}^2\,dt\right\}. \]
In particular, when $p=1$, this and Tonelli's theorem imply that
\begin{align*} 
\mathbb{E}[J_\mathbf{h}(B,B^0)]
	&= \mathbb{E}_B\mathbb{E}_{B^0}[J_\mathbf{h}(B,B^0)] 
	= \mathbb{E}_B[U]
	= 1. 
\end{align*}
When $p>1$, applying Tonelli again and the Cauchy-Schwarz
inequality gives
\[ 
\mathbb{E}[J_\mathbf{h}(B,B^0)^p]
	= \mathbb{E}_B[UV] 
	\le\left(\mathbb{E}_B[U^2]\right)^{1/2}\left(\mathbb{E}_B[V^2]\right)^{1/2}.
\]
For the first factor, we have that
\begin{align*}
\mathbb{E}_B[U^2] 
	&= \exp\left(\frac{1}{2}(p^2-p) \int_0^T
		\|\dot{\mathbf{A}}(t)\|_H^2\,dt\right) 
	\le \exp\left(\frac{1}{2}(p^2-p)
		\|\mathbf{h}\|_{\mathcal{H}_T(\mathfrak{g}_{CM})}^2\right) 
	< \infty.
\end{align*}
For the second factor, first note that
\begin{align*}
\bigg\|\dot{\mathbf{a}}(t) +
		\frac{1}{2}\omega(\mathbf{A}&(t)-2B_t,\dot{\mathbf{A}}(t))
		\bigg\|_\mathbf{C}^2
	\le 2\|\dot{\mathbf{a}}(t)\|_\mathbf{C}^2 +
		2\cdot\frac{1}{4}\|\omega(\mathbf{A(t)}-2B_t,\dot{\mathbf{A}}(t))
		\|_\mathbf{C}^2 \\
	&\le 2\|\dot{\mathbf{a}}(t)\|_\mathbf{C}^2 +
		\frac{1}{2}\|\omega\|_0^2
		\|\mathbf{A}(t)-2B_t\|_W^2\|\dot{\mathbf{A}}(t)\|_W^2 \\
	&\le 2\|\dot{\mathbf{a}}(t)\|_\mathbf{C}^2 +
		\|\omega\|_0^2
		\left(\|\mathbf{A}(t)\|_W^2 + 4\|B\|_{\mathcal{W}_T(W)}^2\right)
		\|\dot{\mathbf{A}}(t)\|_W^2.
\end{align*}
Recall that 
$\|\cdot\|_W\le C\|\cdot \|_H$ for some $C<\infty$ (see for example Theorem
A.1  of \cite{DriverGordina2008}). 
Combining this with the fact that
\[ \|\mathbf{A}(t)\|_H \le \int_0^T \|\dot{\mathbf{A}}(s)\|_H\,ds 
	\le \sqrt{T}\left(\int_0^T\|\dot{\mathbf{A}}(s)\|_H^2\,ds\right)^{1/2}
	\le \sqrt{T}\|\mathbf{h}\|_{\mathcal{H}_T(\mathfrak{g}_{CM})}, \]
implies that
\begin{multline*}
V^2 
	\le \exp\left\{(p^2-p)
		\left(2\|\mathbf{h}\|_{\mathcal{H}_T(\mathfrak{g}_{CM})}^2
		+ C_2^4\|\omega\|_0^2T\|\mathbf{h}\|_{\mathcal{H}_T(\mathfrak{g}_{CM})}^4\right)\right\}
		\\
	\times \exp\left\{4(p^2-p)C_2^2
		\|\omega\|_0^2\|\mathbf{h}\|_{\mathcal{H}_T(\mathfrak{g}_{CM})}^2\|B\|_{\mathcal{W}_T(W)}^2
		\right\}.
\end{multline*}
So letting $\delta_0$ be as in Remark \ref{r.fernique},
$\mathbb{E}_B[V^2]<\infty$ as long as
\[ 4(p^2-p)C_2^2\|\omega\|_0^2\|\mathbf{h}\|_{\mathcal{H}_T(\mathfrak{g}_{CM})}^2 
	< \delta_0, \]
that is, for all $\|\mathbf{h}\|_{\mathcal{H}_T(\mathfrak{g}_{CM})}<\kappa
:=\sqrt{\delta_0/4(p^2-p)C_2^2\|\omega\|_0^2}$.
\end{proof}

In a similar way we may prove the following proposition.

\begin{prop}
\label{p.Jhderp}
Let $p\in[1,\infty)$ and $\mathbf{h}\in\mathcal{H}_T(\mathfrak{g}_{CM})$.  
Then there exists $\varepsilon_0=\varepsilon_0(p)>0$ such that 
\[ \mathbb{E}\left[\sup_{|\varepsilon|\le\varepsilon_0}
	\left|\frac{d}{d\varepsilon}
	J_{\varepsilon\mathbf{h}}(B,B^0)\right|^p\right] <\infty. \]
\end{prop}

\begin{proof}
Note that
\[ J_{\varepsilon\mathbf{h}}
	 = \exp\left(\varepsilon\alpha_1 + \varepsilon^2\alpha_2 +
		\varepsilon^3\alpha_3 + \varepsilon^4 \alpha_4 \right) \]
where
\begin{align} 
\label{e.alpha1}
\alpha_1 = \alpha_1(\mathbf{h}) 
	&= \int_0^T \langle \dot{\mathbf{A}}(t),dB_t\rangle_H + \langle
        \dot{\mathbf{a}}(t)-\omega(B_t,\dot{\mathbf{A}}(t)),dB^0_t\rangle_\mathbf{C} \\ 
\notag
\alpha_2 = \alpha_2(\mathbf{h})
	&= -\frac{1}{2}\int_0^T \|\dot{\mathbf{A}}(t)\|_H^2\,dt 
		+ \frac{1}{2}\int_0^T \langle
		\omega(\mathbf{A}(t),\dot{\mathbf{A}}(t)), dB^0_t\rangle_\mathbf{C}
		\\
\notag
	&\qquad -\frac{1}{2}\int_0^T \|\dot{\mathbf{a}}(t)-
		\omega(B_t,\dot{\mathbf{A}}(t))\|_\mathbf{C}^2 \,dt \\
\notag
\alpha_3 = \alpha_3(\mathbf{h})
	&= -\frac{1}{2}\int_0^T \langle \dot{\mathbf{a}}(t) -
		\omega(B_t,\dot{\mathbf{A}}(t)),\omega(\mathbf{A}(t),
		\dot{\mathbf{A}}(t))\rangle_\mathbf{C}\,dt, \text{ and} \\
\notag
\alpha_4 = \alpha_4(\mathbf{h})
	&= -\frac{1}{8} \int_0^T
	\|\omega(\mathbf{A}(t),\dot{\mathbf{A}}(t)\|_\mathbf{C}^2\,dt.
\end{align}
Thus,
\begin{equation}
\label{e.Jeps}
\frac{d}{d\varepsilon} J_{\varepsilon\mathbf{h}}
 	= J_{\varepsilon\mathbf{h}}\cdot
		(\alpha_1 + 2\varepsilon\alpha_2 +
		3\varepsilon^2\alpha_3 + 4\varepsilon^3 \alpha_4). 
\end{equation}
For fixed $p\in[1,\infty)$, we may choose $\varepsilon_0=\varepsilon_0(p)$ 
sufficiently small that $\varepsilon<\varepsilon_0$ implies
$\varepsilon\|\mathbf{h}\|_{\mathcal{H}_T(\mathfrak{g}_{CM})}<\kappa$, where
$\kappa$ is as given in Proposition \ref{p.Jhp}, and so
$\mathbb{E}[J_{\varepsilon\mathbf{h}}^p]<\infty$.

For the $\alpha_i$'s, note that $\int_0^T \langle \dot{\mathbf{A}},dB\rangle_H$
and
$\int_0^T\langle\omega(\mathbf{A},\dot{\mathbf{A}}),dB^0\rangle_\mathbf{C}$
are Gaussian and hence have finite moments of all orders.  Also, 
\begin{align*} 
\int_0^T \|\dot{\mathbf{a}}(t)
		-\omega(B_t,\dot{\mathbf{A}}&(t))\|_\mathbf{C}^2\,dt 
	\le 2\int_0^T \left(\|\dot{\mathbf{a}}(t)\|_\mathbf{C}^2
		+ \|\omega(B_t,\dot{\mathbf{A}}(t))\|_\mathbf{C}^2 \right)\,dt\\
	&\le 2\int_0^T \left(\|\dot{\mathbf{a}}(t)\|_\mathbf{C}^2
		+ \|\omega\|_0^2\|B\|^2_{\mathcal{W}_T(W)}
		\|\dot{\mathbf{A}}(t))\|_H^2\right)\,dt \\
	&\le 2\left(\|\mathbf{h}\|_{\mathcal{H}_T(\mathfrak{g}_{CM})}^2
		+ \|\omega\|_0^2\|B\|^2_{\mathcal{W}_T(W)}
		\|\mathbf{h}\|_{\mathcal{H}_T(\mathfrak{g}_{CM})}^2\right) \\
	&\le C\left(1 + \|B\|^2_{\mathcal{W}_T(W)}\right), 
\end{align*}
So by Fernique's Theorem (see Remark \ref{r.fernique}) this
term is in $L^p$ for all $p\in[1,\infty)$.  Now if
$N_t:=\int_0^t\langle\dot{\mathbf{a}}-\omega(B,\dot{\mathbf{A}}),dB^0\rangle_\mathbf{C}$,
then $N$ is a martingale and
$\langle N\rangle_T 
	= \int_0^T \|\dot{\mathbf{a}}
		-\omega(B,\dot{\mathbf{A}})\|_\mathbf{C}^2\,dt.$
So by the previous estimate, $\mathbb{E}[\langle N\rangle_T^p]<\infty$ for all
$p\in[1,\infty)$ and hence by the Burkholder-Davis-Gundy inequalities,
$\mathbb{E}|N_T|^p<\infty$. 
Finally, applying the Cauchy-Schwarz inequality and again the previous
estimate implies that
\begin{align}
\label{e.est} 
\int_0^T |\langle \dot{\mathbf{a}}(t) -
		\omega(B_t,\dot{\mathbf{A}}(t)),\omega(\mathbf{A}(t),
		\dot{\mathbf{A}}(t))\rangle_\mathbf{C}|\,dt
	\le C\left(1 + \|B\|^2_{\mathcal{W}_T(W)}\right)
\end{align}
which is again finite by Fernique's theorem. The remaining terms are deterministic and clearly finite.
\end{proof}

\begin{notation}
\label{n.Z}
For $\mathbf{h}_i=(\mathbf{A}_i,\mathbf{a}_i)\in\mathcal{H}_T(\mathfrak{g}_{CM})$, define
\begin{align*}
Z_i &:= Z_{\mathbf{h}_i}(B,B^0) 
	:= \int_0^T \langle \dot{\mathbf{A}}_i(t),dB_t\rangle_H + \langle
        \dot{\mathbf{a}}_i(t)-\omega(B_t,\dot{\mathbf{A}}_i(t)),
		dB_t^0\rangle_\mathbf{C},                           
\end{align*}
\begin{align*}
Z_{ij} &:= Z_{\mathbf{h}_i,\mathbf{h}_j}(B,B^0) 
	:= \int_0^T \langle \omega(\mathbf{A}_j(t),
		\dot{\mathbf{A}}_i(t)),dB_t^0\rangle_\mathbf{C} \\
	&\quad - \int_0^T \bigg[\langle \dot{\mathbf{A}}_i(t),\dot{\mathbf{A}}_j(t)
		\rangle_H 
		+ \langle \dot{\mathbf{a}}_i(t)-\omega(B_t,\dot{\mathbf{A}}_i(t)), \dot{\mathbf{a}}_j(t) 
        - \omega(B_t,\dot{\mathbf{A}}_j(t))\rangle_\mathbf{C} \bigg]\,dt, 
\end{align*}
\begin{align*}
Z_{ijk} := Z_{\mathbf{h}_i,\mathbf{h}_j,\mathbf{h}_k}(B,B^0) 
	&:= -\int_0^T \bigg[
            \langle \dot{\mathbf{a}}_i(t)+\omega(B_t,\dot{\mathbf{A}}_i(t)),
		\omega(\mathbf{A}_k(t),\dot{\mathbf{A}}_j(t))\rangle_\mathbf{C} \\
	&\qquad
            +\langle \dot{\mathbf{a}}_j(t)+\omega(B_t,\dot{\mathbf{A}}_j(t)),
		\omega(\mathbf{A}_k(t),\dot{\mathbf{A}}_i(t))\rangle_\mathbf{C}  \\
	&\qquad  +\langle \dot{\mathbf{a}}_k(t)+\omega(B_t,\dot{\mathbf{A}}_k(t)),
		\omega(\mathbf{A}_j(t),\dot{\mathbf{A}}_i(t))\rangle_\mathbf{C}
		\bigg]\,dt,          
\end{align*}
and
\begin{align*}
Z_{ijkl} &:= Z_{\mathbf{h}_i,\ldots,\mathbf{h}_l} 
	:= - \int_0^T \bigg[
            \langle\omega(\mathbf{A}_l(t),\dot{\mathbf{A}}_i(t)),
		\omega(\mathbf{A}_k(t),\dot{\mathbf{A}}_j(t))\rangle_\mathbf{C} \\
	&\qquad
            +\langle\omega(\mathbf{A}_k(t),\dot{\mathbf{A}}_i(t)),
		\omega(\mathbf{A}_l(t),\dot{\mathbf{A}}_j(t))\rangle_\mathbf{C} \\
	&\qquad
            +\langle\omega(\mathbf{A}_j(t),\dot{\mathbf{A}}_i(t)),
		\omega(\mathbf{A}_l(t),\dot{\mathbf{A}}_k(t))\rangle_\mathbf{C}
		\bigg]\,dt .
  \end{align*}
\end{notation}

The following lemma provides some motivation for Notation \ref{n.Z}.  In
particular, these functions will comprise the factors appearing in the
integration by parts formulae.

\begin{lemma}
\label{l.zDerivs}  
Let $J_\mathbf{h}$ be as given in equation (\ref{e.Jh}) and $Z_i$, $Z_{ij}$, $Z_{ijk}$,
and $Z_{ijkl}$ be as in Notation \ref{n.Z}.  Then
\begin{align}
\label{e.1}
\tag{ {\it i}}
Z_i &= \frac{d}{d\varepsilon}\bigg|_0 J_{\varepsilon \mathbf{h}_i} \\
\label{e.2}
\tag{{\it ii}}
Z_{ij} &= \frac{d}{d\varepsilon}\bigg|_0  Z_i (B-\varepsilon \mathbf{A}_j, B^0-\varepsilon \mathbf{a}_j -
        u_{\varepsilon \mathbf{A}_j}) \\
\label{e.3}
\tag{{\it iii}}
Z_{ijk} &= \frac{d}{d\varepsilon}\bigg|_0  Z_{ij} (B-\varepsilon \mathbf{A}_k, B^0-\varepsilon \mathbf{a}_k -
        u_{\varepsilon \mathbf{A}_k}) \\
\label{e.4}
\tag{{\it iv}}
Z_{ijkl} &= \frac{d}{d\varepsilon}\bigg|_0  Z_{ijk} (B-\varepsilon \mathbf{A}_l, B^0-\varepsilon \mathbf{a}_l -
        u_{\varepsilon \mathbf{A}_l}).
\end{align}
\end{lemma}

\begin{proof}
The lemma follows from simple computations.  For example, recall from equation
(\ref{e.Jeps}) that
\[ \left(\frac{d}{d\varepsilon} J_{\varepsilon\mathbf{h}}\right)
		\bigg|_{\varepsilon=0}
 	= \left(J_{\varepsilon\mathbf{h}}\cdot
		(\alpha_1 + 2\varepsilon\alpha_2 +
		3\varepsilon^2\alpha_3 + 4\varepsilon^3 \alpha_4)\right)
		\bigg|_{\varepsilon=0}
	= \alpha_1,\]
where $\alpha_1=\alpha_1(\mathbf{h})$ is given in (\ref{e.alpha1}).  Taking
$\mathbf{h}=\mathbf{h}_i$ and noting that
$\alpha_1(\mathbf{h}_i)=Z_{\mathbf{h}_i}=Z_i$ completes the proof of \eqref{e.1}.

Similarly, it may be checked that
\begin{equation}
\label{e.Zieps} 
Z_i(B-\varepsilon \mathbf{A}_j,B^0-\varepsilon \mathbf{a}_j-u_{\varepsilon \mathbf{A}_j})
	= Z_i + \varepsilon Z_{ij} + \varepsilon^2\beta_2 +
		\varepsilon^3\beta_3, 
\end{equation}
where
\begin{multline}
\label{e.beta2}
\beta_2 = -\int_0^T\bigg\{
		\frac{1}{2}\langle\dot{\mathbf{a}}_i(t)-\omega(B_t,\dot{\mathbf{A}}_i(t)),\omega(\mathbf{A}_j(t),
		\dot{\mathbf{\mathbf{A}}}_j(t))\rangle_\mathbf{C}	
		\\    
	+ \inprod{\dot{\mathbf{a}}_j(t) - \omega(B_t,\dot{\mathbf{A}}_j(t)),\omega(
		\mathbf{A}_j(t),\dot{\mathbf{A}}_i(t))}_\mathbf{C}\bigg\}\,dt
\end{multline}
and
\begin{equation}
\label{e.beta3}
\beta_3 = -\frac{1}{2}\int_0^T		
	\langle\omega(\mathbf{A}_j(t),\dot{\mathbf{A}}_i(t)),
	\omega(\mathbf{A}_j(t),\dot{\mathbf{A}}_j(t))\rangle_\mathbf{C}\,dt, 
\end{equation}
thus satisfying \eqref{e.2}.  The computations for \eqref{e.3} and \eqref{e.4}
are analogous.
\end{proof}

\begin{prop}
\label{p.Zp}
For all $p\in[1,\infty)$, $\mathbb{E}|Z|^p <\infty$, where $Z$ represents any
element from $\{Z_i,Z_{ij}, Z_{ijk}, Z_{ijkl}:\mathbf{h}_i,\mathbf{h}_j,\mathbf{h}_k,
\mathbf{h}_l\in\mathcal{H}_T(\mathfrak{g}_{CM})\}$.
\end{prop}

\begin{proof}
The integrability of $Z_i=\alpha_1(h_i)$ was already verified in the proof of
Proposition \ref{p.Jhderp}.  The terms in $Z_{ij}$ and $Z_{ijk}$ can be handled similarly
as in that proof, and $Z_{ijkl}$ is deterministic and clearly finite.
\end{proof}

In a similar way to Propositions \ref{p.Jhp} and \ref{p.Jhderp} we may prove the following.
\begin{prop}
\label{p.Zderp}
For any $p\in[1,\infty)$ and
$\mathbf{h}=(\mathbf{A},\mathbf{a})\in\mathcal{H}_T(\mathfrak{g}_{CM})$,  
\[ \mathbb{E}\left[\sup_{|\varepsilon|\le1}
	\left|Z(B-\varepsilon \mathbf{A}, B^0-\varepsilon
		\mathbf{a} - u_{\varepsilon \mathbf{A}})\right|^p\right] <\infty \]
and
\[ \mathbb{E}\left[\sup_{|\varepsilon|\le1}
	\left|\frac{d}{d\varepsilon} Z(B-\varepsilon \mathbf{A}, B^0-\varepsilon
		\mathbf{a} - u_{\varepsilon \mathbf{A}})\right|^p\right] <\infty, \]
where $Z$ represents any element from $\{Z_i,Z_{ij},
Z_{ijk}:\mathbf{h}_i,\mathbf{h}_j,\mathbf{h}_k\in\mathcal{H}_T(\mathfrak{g}_{CM})\}$. 
\end{prop}

\begin{proof}
Recall from equation (\ref{e.Zieps}) that
\[ Z_i(B-\varepsilon \mathbf{A}_j, B^0-\varepsilon
		\mathbf{a}_j - u_{\varepsilon \mathbf{A}_j})
	= Z_i + Z_{ij}\varepsilon + \beta_2\varepsilon^2 +
		\beta_3\varepsilon^3, 
\]
where $\beta_2$ and $\beta_3$ are as given in \eqref{e.beta2} and
\eqref{e.beta3}.  The integrability of $Z_i$ and $Z_{ij}$ follows from
Proposition \ref{p.Zp}, and thus one need only justify the integrability of
$\beta_2$ (as $\beta_3$ is deterministic).  This is easily done using
the polynomial integrability of $\|B\|_{\mathcal{W}_T(W)}$ (compare with
(\ref{e.est})).  Similar arguments work for $Z_{ij}$ and $Z_{ijk}$.
\end{proof}

\begin{notation}
For $m\in\mathbb{N}$, let
\begin{multline*} \Lambda_m := \{\text{partitions } \theta \text{ of }
    \{1,\ldots,m\} : \\
    \theta=\{\gamma^\theta_1,\ldots,\gamma^\theta_{k_\theta}\} \text{ with } \#\gamma^\theta_r\le
    4 \text{ for } r=1,\ldots,k_\theta\}.
\end{multline*}
For $\gamma=\{\ell_1,\ldots,\ell_n\}\in\theta\in\Lambda_m$, we will
always assume that elements are listed in increasing order
$\ell_1<\cdots<\ell_n$.  (Note that $1\le n\le 4$.)
\end{notation}

\begin{notation}
\label{n.Phi}
For any $m\in\mathbb{N}$,
$\gamma=\{\ell_1,\ldots,\ell_n\}\in\theta\in\Lambda_m$, and
$\mathbf{h}_1,\ldots,\mathbf{h}_m\in\mathcal{H}_T(\mathfrak{g}_{CM})$ with $\mathbf{h}_k=(\mathbf{A}_k,\mathbf{a}_k)$, let
$Z_{\gamma} := Z_{\ell_1\cdots\ell_n}$ where the right hand side is as defined
in Notation \ref{n.Z}.  Also let
$\Phi_{\mathbf{h}_1,\ldots,\mathbf{h}_m}=\Phi_{\mathbf{h}_1,\ldots,\mathbf{h}_m}(B,B^0)$
be defined by 
\[\Phi_{\mathbf{h}_1,\ldots,\mathbf{h}_m}
    := \sum_{\theta\in\Lambda_m} Z_{\gamma^\theta_1}\cdots Z_{\gamma^\theta_{k_\theta}}.\]
Further, for $\mathbf{h}_{m+1}\in \mathcal{H}_T(\mathfrak{g}_{CM})$, let 
\[ Z_{\gamma^\theta_j}^{\varepsilon \mathbf{h}_{m+1}}
	:= Z_{\gamma^\theta_j}(B-\varepsilon \mathbf{A}_{m+1},B^0-\varepsilon
		\mathbf{a}_{m+1}-u_{\varepsilon \mathbf{A}_{m+1}}),\]
where $u_\mathbf{A}$ is as defined in (\ref{e.uA}), and
\begin{align*} 
\Phi_{\mathbf{h}_1,\ldots,\mathbf{h}_m}^{\varepsilon \mathbf{h}_{m+1}}
	&:= \Phi_{\mathbf{h}_1,\ldots,\mathbf{h}_m}(B-\varepsilon \mathbf{A}_{m+1},B^0-\varepsilon
		\mathbf{a}_{m+1}-u_{\varepsilon \mathbf{A}_{m+1}}) \\
	&= \sum_{\theta\in\Lambda_m} Z_{\gamma^\theta_1}^{\varepsilon \mathbf{h}_{m+1}}\cdots
		Z_{\gamma^\theta_{k_\theta}}^{\varepsilon \mathbf{h}_{m+1}}.
\end{align*}
\end{notation}

\begin{definition}
Given a normed space $X$ and a function $F:X\rightarrow\mathbb{R}$, we say $F$ 
is {\em polynomially bounded} if
there exist constants $K,M<\infty$ such that
\[ |F(x)|\le 
	K\left(1+ \|x\|_{X}\right)^M \]
for all $x\in X$.
\end{definition}

\begin{definition}
Given $\mathbf{h}\in\mathcal{H}_T(\mathfrak{g}_{CM})$,
we say a function $F:\mathcal{W}_T(G)\rightarrow\mathbb{R}$ {\em
is right $\mathbf{h}$-differentiable} if 
\[ (\hat{\mathbf{h}}F)(\mathbf{g}) 
	:= \frac{d}{d\varepsilon}\bigg|_0 F(\varepsilon \mathbf{h}\cdot \mathbf{g}) \]
exists for all $\mathbf{g}\in\mathcal{W}_T(G)$.  We will say that $F$ is
{\em smooth} if $(\hat{\mathbf{h}}_1\cdots\hat{\mathbf{h}}_mF)(\mathbf{g})$
exists for all $m\in\mathbb{N}$,
$\mathbf{h}_1,\ldots,\mathbf{h}_m\in\mathcal{H}_T(\mathfrak{g}_{CM})$, and 
$\mathbf{g}\in\mathcal{W}_T(G)$.
\end{definition}

\begin{theorem}
\label{t.pathIBP}
Let $m\in\mathbb{N}$ and
$\mathbf{h}_1,\ldots,\mathbf{h}_m\in\mathcal{H}_T(\mathfrak{g}_{CM})$, and suppose that
$F:\mathcal{W}_T(G)\rightarrow\mathbb{R}$ is a smooth function 
such that $F$ and its right derivatives of all orders are polynomially
bounded.  Then
\[ \mathbb{E}\left[(\hat{\mathbf{h}}_1\cdots\hat{\mathbf{h}}_m F)(\xi)\right]
    = \mathbb{E}\left[F(\xi) \Phi_{\mathbf{h}_1,\ldots,\mathbf{h}_m}\right] \]
and $\mathbb{E}|\Phi_{\mathbf{h}_1,\ldots,\mathbf{h}_m}|^p<\infty$ for all $p\in[1,\infty)$.
\end{theorem}

\begin{proof}
That $\Phi_{\mathbf{h}_1,\ldots,\mathbf{h}_m} \in L^p$ for all $p\in[1,\infty)$ follows from the
definition of $\Phi$ and Proposition \ref{p.Zp}, since $L^{\infty-}$ is closed
under products.  Given the integrability results of Propositions
\ref{p.Jhp}, \ref{p.Jhderp}, \ref{p.Zp},
and \ref{p.Zderp},
verifying the integration by
parts is now straightforward. First note that,
if $\hat{\mathbf{h}}F$ is polynomially bounded, then there exist $K,M<\infty$ such that
\begin{equation}
\label{e.est2}
\begin{split}
\sup_{|\varepsilon|\leq 1}
	\left|\frac{d}{d\varepsilon} F(\varepsilon \mathbf{h}\cdot \xi)\right|
    &= \sup_{|\varepsilon|\leq 1}\left|(\hat{\mathbf{h}}F)(\varepsilon \mathbf{h}\cdot \xi)\right|
		\\
	&\le \sup_{|\varepsilon|\leq 1} K\left(1+  
		\|\varepsilon \mathbf{h}\cdot
		\xi\|_{\mathcal{W}_T(\mathfrak{g)}}\right)^M 
    \le  C(\mathbf{h})\left(1+\norm{\xi}_{\mathcal{W}_T(\mathfrak{g)}}\right)^M,
\end{split} 
\end{equation}
where this last expression is integrable by Remark \ref{r.fernique}.

Now consider the $m=1$ case.  This is the
content of Corollary 5.6 of \cite{DriverGordina2008}, but we include it here for
completeness.  By Theorem \ref{t.qi}, we have that
\begin{align*}
\mathbb{E}\left[(\hat{\mathbf{h}}_1F)(\xi)\right]
	&= \mathbb{E}\left[\frac{d}{d\varepsilon}\bigg|_0F(\varepsilon
		\mathbf{h}_1\cdot\xi)\right] 
	= \frac{d}{d\varepsilon}\bigg|_0 \mathbb{E}\left[F(\varepsilon
		\mathbf{h}_1\cdot\xi)\right] \\
	&= \frac{d}{d\varepsilon}\bigg|_0 \mathbb{E}\left[F(\xi)
		J_{\varepsilon\mathbf{h}_1}\right] 
	= \mathbb{E}\left[F(\xi)
		\frac{d}{d\varepsilon}\bigg|_0 J_{\varepsilon\mathbf{h}_1}\right], 
\end{align*}
where the two interchanges of differentiation and integration are justified by
(\ref{e.est2}) and Proposition \ref{p.Jhderp}, respectively.  Then Lemma \ref{l.zDerivs} implies that
\[ \frac{d}{d\varepsilon}\bigg|_0 J_{\varepsilon\mathbf{h}_1} 
	= Z_{\mathbf{h}_1}
	= \Phi_{\mathbf{h}_1}, \]
completing the proof for $m=1$. 

Now, assuming the formula for general $m$, we have that
\begin{align*}
\mathbb{E}\left[(\hat{\mathbf{h}}_1\cdots\hat{\mathbf{h}}_{m+1} F)(\xi)\right]
    &= \mathbb{E}\left[(\hat{\mathbf{h}}_{m+1}F)(\xi)
        \Phi_{\mathbf{h}_1,\ldots,\mathbf{h}_m}(B,B^0)\right] \\
    &= \mathbb{E}\left[\frac{d}{d\varepsilon}\bigg|_0 F(\varepsilon
		\mathbf{\mathbf{h}}_{m+1}\cdot \xi)\Phi_{\mathbf{h}_1,\ldots,\mathbf{h}_m}(B,B^0)\right] \\
    &= \frac{d}{d\varepsilon}\bigg|_0 \mathbb{E}\left[F(\varepsilon
		\mathbf{h}_{m+1}\cdot \xi)\Phi_{\mathbf{h}_1,\ldots,\mathbf{h}_m}(B,B^0)\right] 
\end{align*}
where again we justify the interchange of differentiation and integration by
the estimate in (\ref{e.est2}) above.  Now by Theorem \ref{t.qi}
\begin{align*}
\mathbb{E}[F(\varepsilon
		\mathbf{h}_{m+1}\cdot
		\xi)&\Phi_{\mathbf{h}_1,\ldots,\mathbf{h}_m}(B,B^0)] \\
    &= \mathbb{E}\left[F(\xi)
        \Phi_{\mathbf{h}_1,\ldots,\mathbf{h}_m}(B-\varepsilon \mathbf{A}_{m+1},B^0-\varepsilon
        \mathbf{a}_{m+1}-u_{\varepsilon \mathbf{A}_{m+1}})J_{\varepsilon \mathbf{h}_{m+1}}\right] \\
    &= \mathbb{E}\left[F(\xi)
        \Phi_{\mathbf{h}_1,\ldots,\mathbf{h}_m}^{\varepsilon \mathbf{h}_{m+1}} J_{\varepsilon
		\mathbf{h}_{m+1}}\right].
\end{align*}
Since 
\begin{align*}
\frac{d}{d\varepsilon}
     \Phi_{\mathbf{h}_1,\ldots,\mathbf{h}_m}^{\varepsilon \mathbf{h}_{m+1}} J_{\varepsilon \mathbf{h}_{m+1}}  
	&= \sum_{\theta\in\Lambda_m}\sum_{j=1}^{k_\theta}
		\left(\left(\frac{d}{d\varepsilon}Z_{\gamma^\theta_j}^{\varepsilon \mathbf{h}_{m+1}}\right)		 
		\prod_{l\neq j} Z_{\gamma^\theta_l}^{\varepsilon \mathbf{h}_{m+1}}\right) J_{\varepsilon \mathbf{h}_{m+1}}
          \\
	&\quad + \left(\sum_{\theta\in\Lambda_m}
		\prod_{j=1}^{k_\theta}
		Z_{\gamma^\theta_j}^{\varepsilon \mathbf{h}_{m+1}}\right) 
		\left(\frac{d}{d\varepsilon} 
			J_{\varepsilon \mathbf{h}_{m+1}}\right),
\end{align*}
Propositions \ref{p.Jhp}, \ref{p.Jhderp}, and \ref{p.Zderp} imply that, for all
$p\in[1,\infty)$, there exists $\varepsilon_0>0$ such that
\[ \mathbb{E}\left[ \sup_{|\varepsilon|\le\varepsilon_0} \left|\frac{d}{d\varepsilon}
     \Phi_{\mathbf{h}_1,\ldots,\mathbf{h}_m}^{\varepsilon \mathbf{h}_{m+1}}
	J_{\varepsilon \mathbf{h}_{m+1}}\right|^p\right] <\infty. \]
Thus,
\[ \frac{d}{d\varepsilon}\bigg|_0
	\mathbb{E}\left[F(\xi)
        \Phi_{\mathbf{h}_1,\ldots,\mathbf{h}_m}^{\varepsilon \mathbf{h}_{m+1}} J_{\varepsilon
		\mathbf{h}_{m+1}}\right]
	   = \mathbb{E}\left[F(\xi)\frac{d}{d\varepsilon}\bigg|_0
        \Phi_{\mathbf{h}_1,\ldots,\mathbf{h}_m}^{\varepsilon \mathbf{h}_{m+1}} J_{\varepsilon
		\mathbf{h}_{m+1}}\right]. \]
By Lemma \ref{l.zDerivs},
\begin{align*}
\frac{d}{d\varepsilon}\bigg|_0 \Phi_{\mathbf{h}_1,\ldots,\mathbf{h}_m}^{\varepsilon \mathbf{h}_{m+1}}
	= \sum_{\theta\in\Lambda_m}\frac{d}{d\varepsilon}\bigg|_0
		Z_{\gamma^\theta_1}^{\varepsilon \mathbf{h}_{m+1}} \cdots
		Z_{\gamma^\theta_{k_\theta}}^{\varepsilon \mathbf{h}_{m+1}}
        =  \sum_{\theta\in\Lambda_m}
	\sum_{j=1}^{k_\theta} Z_{\gamma^\theta_j,m+1}
                    \prod_{l\neq j}Z_{\gamma^\theta_l},
  \end{align*}
where, for $\gamma = \{\ell_1,\ldots,\ell_n\}$,
\[ Z_{\gamma,m+1} 
	:= \left\{\begin{array}{ll} 
		Z_{\gamma'} \text{ for } \gamma'=\{\ell_1,\ldots,\ell_n,m+1\} &
			\text{ if } n=1,2,3 \\
		0 & \text{ if } n=4 \end{array} \right. .
\]
Thus, we have that
\begin{align*}
  \frac{d}{d\varepsilon}\bigg|_0
     \Phi_{\mathbf{h}_1,\ldots,\mathbf{h}_m}^{\varepsilon \mathbf{h}_{m+1}} J_{\varepsilon \mathbf{h}_{m+1}}  
    &=  \frac{d}{d\varepsilon}\bigg|_0\Phi_{\mathbf{h}_1,\ldots,\mathbf{h}_m}^{\varepsilon  \mathbf{h}_{m+1}}
            + \Phi_{\mathbf{h}_1,\ldots,\mathbf{h}_m} \frac{d}{d\varepsilon}\bigg|_0 
			J_{\varepsilon \mathbf{h}_{m+1}}  \\
    &= \sum_{\theta\in\Lambda_m}\sum_{\substack{j=1 \\ \#\gamma^\theta_j\leq 3}}^{k_\theta}
		\left(Z_{\gamma^\theta_j,m+1} \prod_{l\neq j} Z_{\gamma^\theta_l}\right) 
            +  \Phi_{\mathbf{h}_1,\ldots,\mathbf{h}_m}Z_{m+1},
\end{align*}
and notice that each term in this sum is a partition of $\{1,\ldots,m,m+1\}$.
In particular, one may see that the final sum is over all of $\Lambda_{m+1}$,
thus yielding the desired expression
$\Phi_{\mathbf{h}_1,\ldots,\mathbf{h}_m,\mathbf{h}_{m+1}}$.
\end{proof}

We conclude this section with the following remark, which gives the
reader some comparison between the integration by parts formula of Theorem
\ref{t.pathIBP} (and indeed the formulae to come in Theorem \ref{t.rHeatKernelIBP}
and Corollary \ref{c.lHeatKernelIBP}) and the usual ``flat'' integration by parts
for Gaussian measures.  In particular, one should think of the functions
$\Phi$ as akin to Hermite functions for the measure
$\nu$.

\begin{remark}
\label{r.hermite}
Let us recall the integration by parts formula for an abstract Wiener space
$(W,H,\mu)$ following from the standard Cameron-Martin theorem.
Let $\{e_i\}_{i=1}^\infty$ be an orthonormal basis of $H$, and
let $\partial_i$ denote the derivative in the direction $e_i$.
Then, for any $k\in\mathbb{N}$, distinct indices $i_1,\ldots,i_k$, and multi-index
$\alpha=(\alpha_1,\ldots,\alpha_k)\in\mathbb{N}^k$, we have
\[ \int_{W} (\partial_{i_1}^{\alpha_1}\cdots\partial_{i_k}^{\alpha_k} f)(w) \,d\mu(w)
	= \int_{W}  f(w) H^\alpha_{i_1,\ldots,i_k}(w) \,d\mu(w) \]
for
$H^\alpha_{i_1,\ldots,i_k}(w) := \prod_{j=1}^k H_{\alpha_j}( \langle
	e_{i_j},w\rangle_H)$,
where $H_n$ are the usual Hermite polynomials and ``$\langle e_i,w\rangle_H$'' is
the Paley-Wiener integral.  

On the other hand, Theorem \ref{t.pathIBP} implies that, for all
$\mathbf{h}_1,\ldots,\mathbf{h}_m\in\mathcal{H}_T(\mathfrak{g}_{CM})$, there
exists $\hat{\Phi}_{\mathbf{h}_1,\ldots,\mathbf{h}_m}\in  L^{\infty-}$ such that
\begin{multline*}
\int_{\mathcal{W}_T(G)} (\hat{\mathbf{h}}_1\cdots\hat{\mathbf{h}}_mF)(\omega)\,d\nu(\omega) 
	= \mathbb{E}\left[(\hat{\mathbf{h}}_1\cdots\hat{\mathbf{h}}_mF)(\xi)\right] \\
	= \mathbb{E}\left[F(\xi)\hat{\Phi}_{\mathbf{h}_1,\ldots,\mathbf{h}_m}(\xi)\right]
	= \int_{\mathcal{W}_T(G)} F(\omega)\hat{\Phi}_{\mathbf{h}_1,\ldots,\mathbf{h}_m}
		(\omega)\,d\nu(\omega).
\end{multline*}
In particular, 
$\hat\Phi_{\mathbf{h}_1,\ldots,\mathbf{h}_m}(\xi) 
	= \mathbb{E}[\Phi_{\mathbf{h}_1,\ldots,\mathbf{h}_m} |\sigma(\xi_t,
		t\in[0,T])]$ a.s., 
and comparing this with the above flat case
leads one to think of $\Phi$ as a polynomial of order $m$ in
\begin{align*} 
\langle \mathbf{h}_i,(B,B^0)\rangle_{\mathcal{H}_T(\mathfrak{g}_{CM})} 
	&:= \int_0^T
		\langle\dot{\mathbf{h}}_i(t),d(B_t,B_t^0)\rangle_{\mathfrak{g}_{CM}} \\
	&= \int_0^T \langle\dot{\mathbf{A}}_i(t),dB_t\rangle_H + \int_0^T
		\langle\dot{\mathbf{a}}_i(t),dB^0_t\rangle_\mathbf{C} 
\end{align*}
as well as additional terms like 
$\int_0^T \langle \omega(B, \dot{\mathbf{A}}_i),dB^0\rangle_\mathbf{C}$.  
The presence of these additional terms of course
follows from the non-commutativity of the setting.  That is, our formula
coincides with the flat case in the event that $\omega\equiv0$.
\end{remark}

\section{Smooth heat kernel measures on $G$}
\label{s.gp}

The smoothness results for the path space measure in the previous section now allow us to
prove smoothness results for the heat kernel measure on $G$.  
For example, in \cite{DriverGordina2008} the path space quasi-invariance was
used to show quasi-invariance for $\nu_T$ under left {\it and right}
translations by elements of the Cameron-Martin subspace; see
Theorem 6.1, Corollary 6.2, and Proposition 6.3 of that reference.

For $g\in G$, let $r_g,\ell_g:G\rightarrow G$ 
denote right and left multiplication by $g$, respectively.  As $G$
is a vector space, to each $g\in G$ we can associate the tangent
space $T_g G$ to $G$ at $g$, which is naturally isomorphic to $G$.
For $h\in\mathfrak{g}$, we define the right and left invariant vector
fields associated to $h$:
\[ \hat{h}(g):=r_{g*}h 
	= \frac{d}{d\varepsilon}\bigg|_0 \varepsilon h\cdot g \quad \text{ and }
\quad \tilde{h}(g) := \ell_{g*}h = \frac{d}{d\varepsilon}\bigg|_0
	g\cdot\varepsilon h, \quad\text{ for all } g\in G. \] 
The vector fields $\hat{h}$ and $\tilde{h}$ act on smooth functions in the
standard way; for example, for $f:G\rightarrow\mathbb{R}$ a Fr\'{e}chet smooth
function on $G$,
\[ (\hat{h}f)(g) 
	= \frac{d}{d\varepsilon}\bigg|_0 f( \varepsilon h\cdot g). \]

\begin{notation}
\label{n.Psi}
Fix $T>0$.  For $m\in\mathbb{N}$, and $h_1,\ldots,h_m\in\mathfrak{g}_{CM}$, let 
$\mathbf{h}_i(t):=\frac{t}{T}h_i\in\mathcal{H}_T(\mathfrak{g}_{CM})$ and 
define $\Psi_{h_1,\ldots,h_m} := \Phi_{\mathbf{h}_1,\ldots,\mathbf{h}_m}$,
where $\Phi$ is as in Notation \ref{n.Phi}.
\end{notation}

\begin{theorem}
\label{t.rHeatKernelIBP}
Fix $T>0$.  Let $m\in\mathbb{N}$, and $h_1,\ldots,h_m\in\mathfrak{g}_{CM}$, and suppose that
$f:G\rightarrow\mathbb{R}$ is a smooth function such that $f$ and
its right derivatives of all orders are polynomially bounded. Then
\[\mathbb{E}\left[(\hat{h}_1\cdots\hat{h}_mf)(\xi_T)\right]
	=\mathbb{E}\left[f(\xi_T)\Psi_{h_1,\ldots,h_m}\right] \]
where $\mathbb{E}|\Psi_{h_1,\ldots,h_m}|^p<\infty$ for all $p\in[1,\infty)$.
\end{theorem}

\begin{proof}
Clearly, the integrability of $\Phi$ proved in Theorem \ref{t.pathIBP} and the
definition of $\Psi$ imply that $\Psi_{h_1,\ldots,h_m}\in
L^p$ for all $p\in[1,\infty)$.
The integration by parts also follows from Theorem \ref{t.pathIBP}.  To see this,
let $F:\mathcal{W}_T(G)\rightarrow \mathbb{R}$ be given by
$F(\mathbf{g})=f(\mathbf{g}(T))$ 
and 
$\mathbf{h}_i(t):=\frac{t}{T}h_i\in\mathcal{H}_T(\mathfrak{g}_{CM})$.  Now note that
\begin{align*}
\mathbb{E}\left[(\hat{h}_1\cdots\hat{h}_mf)(\xi_T)\right]
	&= \mathbb{E}\left[\frac{d}{d\varepsilon_1}\bigg|_0\cdots\frac{d}{d\varepsilon_m}\bigg|_0
		f(\varepsilon_m h_m\cdot(\cdots(\varepsilon_1 h_1\cdot \xi_T))\right] \\
	&= \mathbb{E}\left[\frac{d}{d\varepsilon_1}\bigg|_0
		\cdots\frac{d}{d\varepsilon_m}\bigg|_0
		F(\varepsilon_m \mathbf{h}_m\cdot(\cdots(\varepsilon_1
		\mathbf{h}_1\cdot \xi))\right] \\
	&= \mathbb{E}\left[(\hat{\mathbf{h}}_1\cdots\hat{\mathbf{h}}_mf)(\xi_T)\right]
	= \mathbb{E}\left[F(\xi)\Phi_{\mathbf{h}_1,\ldots,\mathbf{h}_m}\right] \\
	&= \mathbb{E}\left[f(\xi_T)\Psi_{h_1,\ldots,h_m}\right].
\end{align*}
\end{proof}

\begin{remark}
\label{r.barPsi}
As in the path measure case (see Remark \ref{r.hermite}), 
Theorem \ref{t.rHeatKernelIBP} implies that, for all $h_1,\ldots,h_m\in\mathfrak{g}_{CM}$, 
there exists
$\hat{\Psi}_{h_1,\ldots,h_m}\in L^{\infty-}(\nu_T)$ such that
\begin{align*} 
\int_G (\hat{h}_1\cdots\hat{h}_mf)(g)\,d\nu_T(g)
	= \int_G f(g)\hat{\Psi}_{h_1,\ldots,h_m}(g)\,d\nu_T(g),
\end{align*}
where
$\hat{\Psi}_{h_1,\ldots,h_m}(\xi_T) 
	= \mathbb{E}[\Psi_{h_1,\ldots,h_m}\mid\sigma(\xi_T)]$ a.s.
\end{remark}

\begin{corollary}
\label{c.lHeatKernelIBP}
Under the hypotheses of Theorem \ref{t.rHeatKernelIBP},
\[ \mathbb{E}[(\tilde{h}_1\cdots\tilde{h}_mf)(\xi_T)]
	=\mathbb{E}[f(\xi_T)\tilde{\Psi}_{h_1,\ldots,h_m}(\xi_T)], \]
where
\[
\tilde{\Psi}_{h_1,\ldots,h_m}(g) 
	:= (-1)^m\hat{\Psi}_{h_1,\ldots,h_m}(g^{-1}).
\]
and $\hat{\Psi}$ is as in Remark \ref{r.barPsi}.

\end{corollary}
\begin{proof}
Take $u(g):=f(g^{-1})=f(-g)$.  We proceed by induction.  The $m=1$ case is
proved in Corollary 6.5 of
\cite{DriverGordina2008}, but we include the proof here for completeness.  Note
first that, for any $g\in G$ and $h\in \mathfrak{g}_{CM}$,
\begin{equation}
\label{e.a} 
(\tilde{h}f)(g) = \frac{d}{d\varepsilon}\bigg|_0 f(g\cdot\varepsilon h)
	= \frac{d}{d\varepsilon}\bigg|_0 u(-\varepsilon h\cdot g^{-1})
	= -(\hat{h}u)(g^{-1}). 
\end{equation}
Thus, making repeated use of equation (\ref{e.inv}), we have that
\begin{align*}
\mathbb{E}[(\tilde{h}f)(\xi_T)] 
	&= -\mathbb{E}[(\hat{h} u)(\xi_T^{-1})] 
	= -\mathbb{E}[(\hat{h} u)(\xi_T)] \\
	&= -\mathbb{E}[u(\xi_T)\hat{\Psi}_h(\xi_T)]
	= -\mathbb{E}[f(\xi_T^{-1})\hat{\Psi}_h(\xi_T)] \\
	&= -\mathbb{E}[f(\xi_T)\hat{\Psi}_h(\xi_T^{-1})],
\end{align*}
where we have applied Theorem \ref{t.rHeatKernelIBP} in the third equality.
Now assuming the formula for $m$ and again using equations (\ref{e.a}) and
(\ref{e.inv}) and Theorem \ref{t.rHeatKernelIBP} gives
\begin{align*}
\mathbb{E}&\left[(\tilde{h}_1\cdots\tilde{h}_{m+1}f)(\xi_T)\right]
	= (-1)^m\mathbb{E}\left[(\tilde{h}_{m+1}f)(\xi_T)
		\hat{\Psi}_{h_1,\ldots,h_m}(\xi_T^{-1})\right] \\
	&= (-1)^{m+1}\mathbb{E}\left[ (\hat{h}_{m+1}u)(\xi_T^{-1})
		\hat{\Psi}_{h_1,\ldots,h_m}(\xi_T^{-1})\right] \\
	&= (-1)^{m+1}\mathbb{E}\left[(\hat{h}_{m+1}u)(\xi_T)
		\hat{\Psi}_{h_1,\ldots,h_m}(\xi_T)\right] 
	= (-1)^{m+1}\mathbb{E}\left[ (\hat{h}_{m+1}u)(\xi_T)
		\Psi_{h_1,\ldots,h_m}\right]  \\
	&= (-1)^{m+1}\mathbb{E}\left[ u(\xi_T) \Psi_{h_1,\ldots,h_{m+1}}\right] 
	= (-1)^{m+1}\mathbb{E}\left[u(\xi_T)
		\hat{\Psi}_{h_1,\ldots,h_{m+1}}(\xi_T)\right] \\
	&= (-1)^{m+1}\mathbb{E}\left[f(\xi_T^{-1})
		\hat{\Psi}_{h_1,\ldots,h_{m+1}}(\xi_T)\right] 
	= (-1)^{m+1}\mathbb{E}\left[f(\xi_T)
		\hat{\Psi}_{h_1,\ldots,h_{m+1}}(\xi_T^{-1})\right].
\end{align*}
\end{proof}

\providecommand{\bysame}{\leavevmode\hbox to3em{\hrulefill}\thinspace}
\providecommand{\MR}{\relax\ifhmode\unskip\space\fi MR }
\providecommand{\MRhref}[2]{%
  \href{http://www.ams.org/mathscinet-getitem?mr=#1}{#2}
}
\providecommand{\href}[2]{#2}

\end{document}